\documentclass[a4paper]{article}

\usepackage[margin=1in]{geometry}

\usepackage{
amsmath,
amsthm,
amscd,
amssymb,
stmaryrd,
}
\usepackage{comment}
\usepackage{tikz}
\usetikzlibrary{positioning,arrows,calc}
\usetikzlibrary{decorations.pathreplacing}
\usepackage{xspace}

\usepackage{graphicx}

\usepackage{url}

\theoremstyle{plain}
\newtheorem{lemma}{Lemma}
\newtheorem{definition}{Definition}
\newtheorem{corollary}{Corollary}

\newtheorem{theorem}{Theorem}
\newtheorem{conjecture}{Conjecture}

\newtheoremstyle{derp}
{3pt}
{3pt}
{}
{}
{\upshape}
{:}
{.5em}
{}
\theoremstyle{derp}

\newcommand{\Z}{\mathbb{Z}}

\newcommand{\N}{\mathbb{N}}

\newcommand{\ID}{\mathrm{id}}

\newcommand{\Homeo}{\mathrm{Homeo}}

\newcommand{\End}{\mathrm{End}}
\newcommand{\Aut}{\mathrm{Aut}}

\newcommand{\llb}{\llbracket}
\newcommand{\rrb}{\rrbracket}

\title{Graph and wreath products in topological full groups of full shifts}

\author{
Ville Salo \\
vosalo@utu.fi
}

\begin{document}
\maketitle

\begin{abstract}
We prove that the topological full group $\llb X \rrb$ of a two-sided full shift $X = \Sigma^\Z$ contains every right-angled Artin group (also called a graph group). More generally, we show that the family of subgroups with ``linear look-ahead'' is closed under graph products. We show that the lamplighter group $\Z_2 \wr \Z$ embeds in $\llb X \rrb$, and conjecture that it does not embed in $\llb X \rrb$ with linear look-ahead. Generalizing the lamplighter group, we show that whenever $G$ acts with ``unique moves'' (or at least ``move-$A$ithfully''), we have $A \wr G \leq \llb X \rrb$ for finite abelian groups $A$. We show that free products of finite and cyclic groups act with unique moves. We show that $\Z^2$ does not admit move-$A$ithful actions, and conjecture that $\Z_2 \wr \Z^2$ does not embed in $\llb X \rrb$ at all. We show that topological full groups of all infinite nonwandering sofic shifts have the same subgroups, and that this set of groups is closed under commensurability. The group $\llb X \rrb$ embeds in the higher-dimensional Thompson group $2$V, so it follows that $2$V contains all RAAGs, refuting a conjecture of Belk, Bleak and Matucci.
\end{abstract}

\section{Introduction}

We prove that all right-angled Artin groups (RAAGs) embed in the topological full group of the two-sided binary full shift. Write $X = \{0,1\}^\Z$ and write $\llb X \rrb$ for this group. We prove a weak closure property for graph products of subgroups of $\llb X \rrb$: for a subgroup of $\llb X \rrb$, we define its \emph{look-ahead} function, and prove that any graph product of groups with linear look-ahead (with possibly different slopes) also embeds in $\llb X \rrb$ with bounded look-ahead. Bounded (resp.\ linear) look-ahead means roughly that for each element $g$ of the group there is a configuration $x$ which $g$ decides to shift by at least $n \geq 1$ steps, after looking at only $n + C$ (resp.\ $Cn + C$) coordinates of $x$.

Embeddability of RAAGs, together with the closure of the class of embeddable subgroups under commensurability (which we also show), implies the embeddability of many other groups such as surface groups and Coxeter groups, see \cite{BeBlMa20}. The embeddability of $F_2 \times F_2$ in itself implies many things, in particular that there exist finitely-generated subgroups with undecidable conjugacy problem.

A key idea, conveyor belts \cite{Sa16a}, is borrowed from the theory of automorphism groups of subshifts. This is roughly equal to the bucket-passing of \cite{KiRo90}. This is simply a technique for ``drawing'' configurations of one subshift on the configurations of another, by a shift-invariant continuous rule, and then acting on the ``simulated'' configurations. Kim and Roush used this idea in \cite{KiRo90} to show that the automorphism groups of full shifts with different alphabets embed in each other. In the same paper they show that finite graph products of finite groups, thus also finitely-generated RAAGs, embed in $\Aut(\Sigma^\Z)$ for any alphabet $\Sigma$. We proved closure of the set of subgroups under free products and graph products using conveyor belts respectively in \cite{Sa16a,Sa20}. Much of the present paper is analogous to \cite{Sa20}, although the exact results and proof details differ.

The group $\llb X \rrb$ is interesting because it embeds in various groups that are defined by a word rewriting action. In particular it embeds in the Brin-Thompson group $2$V. It is known that Thompson's group V contains a copy of a RAAG $G$ if and only if $\Z^2 * \Z \not\leq G$ \cite{CoHa16}. In light of this, it was conjectured in \cite{BeBlMa20} that $\Z^d * \Z$ is the correct obstruction for embeddability in the Brin-Thompson group $(d-1)$V defined in \cite{Br04}. In \cite{BeBlMa20}, it was shown that at least RAAGs with $m$ nodes and $n$ non-commuting relations embed in $(m+n)$V. In \cite{Ka18}, it was proved that if $G$ is a RAAG with $\Z^d * \Z \not\leq G$, then indeed $G \leq (d-1)$V, as predicted by the conjecture. In this paper, we refute the conjecture from \cite{BeBlMa20} and completely settle the issue of RAAG embeddability: all (countable) RAAGs embed in $n$V for all $n \geq 2$.

The group $\llb X \rrb$ also embeds in the group of Turing machines $\mathrm{RTM}(2,1)$ from \cite{BaKaSa16}, and is isomorphic to the group of finite-state machines $\mathrm{RFA}(2,1)$ considered there. It also embeds in the group of cellular automata $\Aut(X \times X)$ (the group of shift-commuting homeomorphisms $f : X\times X \to X \times X$) by translating movements of the tape into permutations of the bits on the second component of the product subshift. It of course does not embed in Thompson's group V (for example because it contains strictly more RAAGs than V).

We do not know which groups embed in $\llb X \rrb$ with linear look-ahead. We prove that the lamplighter group $\Z_2 \wr \Z$ embeds in $\llb X \rrb$, and conjecture that it does not embed with linear look-ahead. We are far from showing this, indeed we are not even able to show that $\llb X \rrb$ does not abstractly embed in itself with linear look-ahead.

Our proof that $\Z_2 \wr \Z \leq \llb X \rrb$ is based on the fact that $\Z$ admits an action by $\llb X \rrb$ with unique moves, roughly meaning that in any finite set we can find an element with unique cocycle value on some configuration (different from the values of cocycles of all other elements in the set). We show that all countable free products of finite and cyclic groups admit such actions. A weaker condition for embedding $A \wr G$ is the $G$ admits a ``move-$A$ithful'' action, analogous to the more general notion of $A$ithfulness studied in \cite{Sa20}. We show that $\Z^2$ does not act with unique moves (or even move-$A$ithfully), and conjecture that $\Z_2 \wr \Z^2$ indeed does not embed in $\llb X \rrb$ at all. 

As a side-effect of our proof of closure of the set of subgroups of $\llb X \rrb$ under commensurability, we prove that this set of subgroups is the same if $X$ is replaced by any infinite nonwandering sofic shift.

Topological full groups are most commonly studied in the context of minimal subshifts. We note that no non-abelian RAAGs embed in the topological full group of a minimal subshift, since the latter groups are amenable \cite{JuMo13}. On the other hand our groups do not contain $\llb Y \rrb$ for minimal subshifts $Y$, since $\llb \Sigma^\Z \rrb$ is residually finite and $\llb Y \rrb$ is simple. The groups $\llb \Sigma^\Z \rrb$ also do not have finitely-generated commutator subgroups (for reasons that are not particularly deep).

Finally, we note that topological full groups of one-sided subshifts of finite type have been previously studied in \cite{Ma12}. They are more related to Thompson's V than our group $\llb X \rrb$; indeed $\llb \{0,1\}^\N \rrb$ is isomorphic to Thompson's V. These groups do not (always) embed in our group, as they can contain infinite simple groups, and our group does not embed in at least one of them (namely Thompson's V).

\section{Definitions}

An \emph{alphabet} is a finite set $\Sigma$ with at least two elements.
The \emph{full shift} is the set $\Sigma^\Z$ where $\Sigma$ is an alphabet.
The \emph{(left) shift} is the map $\sigma : \Sigma^\Z \to \Sigma^\Z$ defined by $\sigma(x)_i = x_{i+1}$, and it makes the full shift into a dynamical system. A \emph{shift} of a configuration $x \in \Sigma^\Z$ is $\sigma^i(x)$ for some $i \in \Z$.
A \emph{subshift} is a closed shift-invariant subset of a full shift.
A configuration $x \in X$ where $X$ is a subshift has \emph{period} $p \geq 1$ if $\sigma^p(x) = x$. (We do not need the concept of a least period.) A configuration is \emph{aperiodic} if it has no period. 
We denote two-sided configurations in a full shift by $x.y \in \Sigma^\Z$, where $x \in \Sigma^{(-\infty, -1]}, y \in \Sigma^\N$, so the coordinate immediately to the right of the decimal point is coordinate $0$.
We concatenate words $u,v$ by simply writing $uv$.
Write $[w]_i = \{x \in X \;|\; x_{[i, i+|w|-1]} = w \}$, where $X$ is a subshift clear from context (usually a full shift).
A set of words $W \subset A^*$ is \emph{mutually unbordered} if $ut, tv \in W \implies |t| = 0 \vee |u| = |v| = 0$. A word $w$ is \emph{unbordered} if $\{w\}$ is mutually unbordered.

\begin{definition}
The topological full group of a subshift $X$, which we denote by $\llb X \rrb$, is the group of all homeomorphisms $f : X \to X$ such that there exists a continuous function $c : X \to \Z$ such that $f(x) = \sigma^{c(x)}(x)$ for all $x \in X$.
\end{definition}

The function $c : X \to \Z$ in the definition of $\llb X \rrb$ is called the \emph{cocycle}. Usually one writes $\llb \sigma \rrb$, as the topological full group can be associated to any homeomorphism $\sigma$. However, the dynamics usually stays the same (shift) for us while the set changes, so we just write the set inside the brackets.

\begin{definition}
The Brin-Thompson $2V$ is, for $X = \{0,1\}^\Z$, the subgroup of $\Homeo(X)$ containing all maps $f : X \to X$ such that there exists $n \in \N$ and $c : \Sigma^{2n+1} \to \Sigma^* \times \Sigma^*$, such that for all $u \in \Sigma^n, v \in \Sigma^{n+1}$ with $c(u,v) = (u',v')$ we have $f(xu.vy) = xu'.v'y$.
\end{definition}

The way the definitions are stated, the following is obvious.

\begin{lemma}
$\llb \{0,1\}^\Z \rrb \leq 2\mathrm{V}$.
\end{lemma}

\begin{definition}
Let $X$ be a subshift and let $G \leq \llb X \rrb$ be a subgroup. We say $G$ has \emph{look-ahead} $\alpha : \N \to \N$ if $\alpha(0) = 0$ and for all $g \in G$ with cocycle $c$, there exists $n \geq 1$ and a cylinder $[w]_{-n}$ with $w \in \Sigma^{2n+1}$ such that $\forall x \in [w]_{-n}: c(x) = m$, with $|m| + \alpha(|m|) \geq n$. We say $G$ has \emph{plook-ahead} (short for periodic look-ahead) $\alpha : \N \to \N$ if for all $g \in G$ with cocycle $c$ there exists a periodic point $x$ with period at most $2n+1$ such that $|c(x)| + \alpha(|c(x)|) \geq n \geq 1$.
\end{definition}

Note that because $\alpha(0) = 0$, we have $c(x) \neq 0$ for the configurations $x \in [w]_{-n}$ in the definition of look-ahead, and similarly $x$ cannot be a fixed point in the definition of plook-ahead.

From now on let us restrict to full shifts $X = \Sigma^\Z$, unless indicated otherwise.

\begin{lemma}
\label{lem:lookaplooka}
If $G$ has look-ahead $\alpha$, then it has plook-ahead $\alpha$.
\end{lemma}

\begin{proof}
Let $g \in G$ with cocycle $c$. By the look-ahead assumption, there exist $n \geq 1$ and a cylinder $[w]_{-n}$ with $w \in \Sigma^{2n+1}$ such that $\forall x \in [w]_{-n}: c(x) = m$, with $|m| + \alpha(|m|) \geq n$. The periodic point $x = \sigma^n(w^\Z) \in [w]_{-n}$ has period at most (exactly) $2n+1$ and $|c(x)| + \alpha(|c(x)|) = |m| + \alpha(|m|) \geq n$
\end{proof}

A function $\alpha : \N \to \N$ is \emph{linear} if there exists $c \in \N$ such that $\forall n: \alpha(n) \leq cn+c$.

\begin{definition}
\label{def:lookahead}
Let $\mathcal{G}_{\mathrm{lin}}$ be the class of groups that embed in $\llb \Sigma^\Z \rrb$ with linear plook-ahead, for some alphabet $\Sigma$, $\mathcal{G}_{\mathrm{bnd}}$ the groups that embed with bounded look-ahead. 
Write $\mathcal{G}_{\mathrm{plin}}$ and $\mathcal{G}_{\mathrm{pbnd}}$ 
for the corresponding classes defined with plook-ahead. 
For each of the four classes, $\mathcal{G}_x$, write $\mathcal{G}_{x,\Sigma}$ for the restriction to embeddings in $\llb \Sigma^\Z \rrb$. 
\end{definition}

By Lemma~\ref{lem:lookaplooka}, the look-ahead classes are included in the corresponding plook-ahead classes. We will see in the following section that all these classes are the same, in the sense that they contain the same isomorphism classes of groups.

\section{Graph products and RAAGs}

Recall that all our alphabets have finite cardinality at least two.

\begin{theorem}
For every alphabet $\Sigma$, every RAAG embeds in $\llb \Sigma^\Z \rrb$.
\end{theorem}

This includes the case of RAAGs defined by a countably infinite graph. The following theorem is a robustness property for the various look-ahead properties.

\begin{theorem}
For any alphabet $\Sigma$,
$ \mathcal{G}_{\mathrm{bnd},\Sigma} = \mathcal{G}_{\mathrm{pbnd},\Sigma} = \mathcal{G}_{\mathrm{lin},\Sigma} = \mathcal{G}_{\mathrm{plin},\Sigma} = \mathcal{G}_{\mathrm{bnd}} = \mathcal{G}_{\mathrm{pbnd}} = \mathcal{G}_{\mathrm{lin}} = \mathcal{G}_{\mathrm{plin}} $.
\end{theorem}

Due to this theorem, we write this class as simply $\mathcal{G}$. These theorems are direct consequences of Theorem~\ref{thm:GraphProducts} below. We need a simple lemma. 

\begin{lemma}
\label{lem:two}
If $G \in \mathcal{G}_{\mathrm{plin}}$ then $G$ admits an action with linear plook-ahead where the cocycle of every group element takes only even values.
\end{lemma}

\begin{proof}
Simply ignore the odd coordinates: apply the cocycle to the subsequence of symbols at even positions, and multiply its values by $2$. Clearly we still have linear plook-ahead, using the same periodic points, but inserting a fixed symbol in every odd coordinate.
\end{proof}

\begin{theorem}
\label{thm:GraphProducts}
Any countable graph product of groups in $\mathcal{G}_{\mathrm{plin}}$ is in $\mathcal{G}_{\mathrm{bnd, \Sigma}}$ for any alphabet $\Sigma$.
\end{theorem}

\begin{proof}
We show that if $G$ is a countable graph product of groups in $\mathcal{G}_{\mathrm{plin}}$, then $G \in \mathcal{G}_{\mathrm{lin, \Sigma}}$, and if the product is finite then $G \in \mathcal{G}_{\mathrm{bnd, \Sigma}}$. To get countable graph products $G$ in $\mathcal{G}_{\mathrm{bnd, \Sigma}}$, one can simply apply the theorem again, seeing $G$ as a one-node graph product.

We start with a notion of ``simulation'' for topological full group elements. Let $\Theta$ be a finite alphabet, $s : X \to \Theta$ is a continuous map, $f \in \llb X \rrb$, $Y = \Theta^\Z$, and $g \in \llb Y \rrb$ has cocycle $c$. We define $\frac{g}{f, s} \in \llb X \rrb$ as follows: Define a morphic image for the dynamical system $(X, f)$ by recording the $s$-images over the $f$-orbit, i.e.\ $\pi : X \to \Theta^\Z$, $\pi(x)_i = s(f^i(x))$. If $x \in X$, set $\frac{g}{f,s}(x) = f^{c(\pi(x))}(x)$.

We give another description of this operation, which is how we use this in practice. A topological full group element $f$ is determined by its cocycle $c : X \to \Z$. Applying the cocycle to every shift of a configuration, we obtain an \emph{enriched cocycle}, namely the shift-commuting continuous function $c' : X \to \Z^\Z$ where $c'(x)_i = c(\sigma^i(x))$. Taking $\Z$ as nodes, and writing a directed edge from $n$ to $n + c'(x)_n$ we obtain a directed graph whose edges indicate how the configuration moves inside its orbit, more precisely if we imagine the origin at a node, then the forward edge tells us the origin after applying $f$. We sometimes call this ``imagined origin'' the \emph{head}.

Following the edges forward in this graph, the configuration splits into cycles and bi-infinite paths. Similarly, applying $s$ at every position we obtain a map $s' : X \to Y$, $s'(x)_i = s(\sigma^i(x))$. Now, the application of $\frac{g}{f,s}$ simply reads the ``simulated'' configuration (in the symbols in $s'(x)$) on the cycle or path that the origin is part of, applies to cocycle of $g$ to that simulated configuration, and finally applies $f$ as if it were the shift map on this cycle or path.

Now, let the graph defining a graph product be $(V,E)$, with groups $G_u \in \llb \Sigma_u^\Z \rrb$ at the countably many nodes $u \in V$. Pick an order $<$ on $V$ with order type $\omega$. By Lemma~\ref{lem:two} we may assume that the cocycle of each $g \in G_u$, for all $u \in V$, takes only even values. To each $g \in G_u$ we associate a topological full group element $\hat g \in \llb X \rrb$, and we will always take $\hat g = \frac{g}{f_u,s_u}$ for suitable $f_u, s_u$.

We will describe $f_u$ and $s_u$ simultaneously for all $u \in V$. For this, we associate to $x \in X$ an auxiliary graph structure where for each $u$ we have a set of edges of color $u$. The subgraph of edges with color $u$ splits into cycles and bi-infinite paths, and isolated nodes. Equivalently, the in- and out-degrees of each node are equal to each other, and both either $0$ or $1$. For the purpose of simulation, we imagine every node is completed afterward to have in- and out-degree $1$ for color-$u$ edges, by adding a self-loop. The action of $f_u$ is to simply follow the forward edges of color $u$, as in the description of simulation in terms of graphs above.

Nodes are also colored, and a node can have up to two colors, and for each color it carries a symbol: if a node is of color $u$, then it carries a \emph{$u$-symbol} $a \in \Sigma_u$. The map $s_u$ simply outputs the $u$-symbol of the node at the current position if the node has color $u$, and otherwise outputs a fixed arbitrary symbol in $\Sigma_u$ (which will not have any effect).

Once we have defined this directed vertex- and edge-colored graph (with additional symbols at some nodes), the choices of $f_u$ and $s_u$ imply that the action of $\hat g$ for $g \in G_u$ simply interprets the configuration written in the $u$-symbols on the cycle or bi-infinite path that the origin is part of, in the graph formed by $u$-edges, and moves around this cycle or path according to its cocycle.

It now remains to describe the graph. Suppose the group $G_u$ has linear look-ahead function $\alpha_u(n) \leq c_u n + c_u$, where we may assume $c_u \geq 1$. For each $(a,b) \in \{1, ..., 100 c_u\}^2$ and word $v \in \Sigma_u^{a+b}$ we pick a word $w_{u,a,b,v} \in \Sigma^*$, so that these words are mutually unbordered over all $u, a, b, v$, and $|w_{u,a,b,v}| \geq a + b$. Let $W_u = \{w_{u,a,b,c} \;|\; a, b, v\}$. Let $t'$ be mutually unbordered with all these words.

A maximal segment of consecutive words from $W_u = \{w_{u, a, b, v} \;|\; a, b, v\}$ (possibly infinite in one or two directions) is called a \emph{type-$u$ conveyor belt}. We call the words $w_{u, a, b, v}$ forming a conveyor belt its \emph{blocks}. The first $a$ positions in such a block (appearing in a configuration $x \in B^\Z$) are called the \emph{top precells} and the following $b$ positions the \emph{bottom precells}.

We now define the \emph{top cells} and \emph{bottom cells} of a block in a conveyor belt of type $u$, by slightly perturbing the precells. If $u > u'$ and $\{u,u'\} \notin E$, and a block $w_{u,a,b,v}$ appears immediately to the left of a $w_{u',a',b',v'}$-block then the $a$th bottom cell of the $w_{u,a,b,v}$-block coincides with the leftmost top precell of that $w_{u',a',b',v'}$-block. Similarly, if $u > u'$ and $\{u,u'\} \notin E$, and a block $w_{u,a,b,v}$ appears immediately to the right of a $w_{u',a',b',v'}$-block then the leftmost top cell of the $w_{u,a,b,v}$-block coincides with the $b'$th bottom precell of that $w_{u',a',b',v'}$-block. Otherwise the $i$th top (resp.\ bottom) cell of any block is equal to its $i$th top (resp.\ bottom) precell. A cell that is both the top and bottom cell of a block (necessarily of conveyor belts of different types) is called a \emph{shared cell}.

Now, the auxiliary graph has nodes $\Z$, a node has color $u$ if it is a top or bottom cell of a type-$u$ conveyor belt. The edges of color $u$ connect each type-$u$ conveyor belt into a cycle, bi-infinite path or two bi-infinite paths, as follows: in a single conveyor belt, first connect the top cells from left to right, and in consecutive blocks $I, J$ ($I$ immediately to the left of $J$) connect the rightmost top cell of $I$ to the leftmost top cell of $J$. Do the same for the bottom cells, in inverse direction, so all edges go right-to-left. Finally, at the ends of a conveyor belt join the ends of paths together, e.g.\ if immediately to the left of a $u$-block $I$ there is no $u$-block, then connect the leftmost bottom cell to the leftmost top cell. Any finite conveyor belt gives an cycle, which, imagining the top cells on top of the bottom cells, looks like a conveyor belt.

Doing this for all $u$, we have specified the nodes and edges, and their colors $u$. Finally, the $u$-symbol on a $u$-colored vertex is read from the word $v$ in $w_{u,a,b,v}$. This construction is illustrated in Figure~\ref{fig:Construction}.

\begin{figure}
\begin{center}
\begin{tikzpicture}[scale=0.35]
\node () at (0,0) {0};
\node () at (1,0) {1};
\node () at (2,0) {1};
\node () at (3,0) {1};
\node () at (4,0) {0};
\node () at (5,0) {1};
\node () at (6,0) {1};
\node () at (7,0) {0};
\node () at (8,0) {0};
\node () at (9,0) {0};
\node () at (10,0) {1};
\node () at (11,0) {0};
\node () at (12,0) {0};
\node () at (13,0) {1};
\node () at (14,0) {1};
\node () at (15,0) {0};
\node () at (16,0) {0};
\node () at (17,0) {0};
\node () at (18,0) {1};
\node () at (19,0) {0};
\node () at (20,0) {1};
\node () at (21,0) {1};
\node () at (22,0) {0};
\node () at (23,0) {0};
\node () at (24,0) {0};
\node () at (25,0) {0};
\node () at (26,0) {1};
\node () at (27,0) {0};
\node () at (28,0) {1};
\node () at (29,0) {1};
\node () at (30,0) {0};
\node () at (31,0) {0};
\node () at (32,0) {0};
\node () at (33,0) {0};
\node () at (34,0) {0};
\node () at (35,0) {1};
\node () at (36,0) {0};
\node () at (37,0) {1};
\node () at (38,0) {1};
\node () at (39,0) {1};
\node () at (40,0) {0};
\draw [decorate,decoration={brace,amplitude=10pt},xshift=0,yshift=1]
(4.5,0.8) -- (11.5,0.8) node [black,midway,yshift=0.6cm] 
{\footnotesize $w_{1,1,3,0111}$};
\draw [decorate,decoration={brace,amplitude=10pt},xshift=0,yshift=1]
(12.5,0.8) -- (19.5,0.8) node [black,midway,yshift=0.6cm] 
{\footnotesize $w_{1,1,3,0111}$};
\draw [decorate,decoration={brace,amplitude=10pt},xshift=0,yshift=1]
(19.5,0.8) -- (27.5,0.8) node [black,midway,yshift=0.6cm] 
{\footnotesize $w_{1,2,2,1100}$};
\draw [decorate,decoration={brace,amplitude=10pt},xshift=0,yshift=1]
(27.5,0.8) -- (36.5,0.8) node [black,midway,yshift=0.6cm] 
{\footnotesize $w_{0,2,3,10101}$};
\draw[thick] (4.5,3) -- (4.5,-8);
\draw[thick] (11.5,3) -- (11.5,-8);
\draw[thick] (12.5,3) -- (12.5,-8);
\draw[gray] (19.5,3) -- (19.5,-8);
\draw[thick] (27.5,3) -- (27.5,-8);
\draw[thick] (36.5,3) -- (36.5,-8);
\node[circle,inner sep =0.4] (a) at (5,-3) {\footnotesize $0$};
\node[circle,inner sep =0.4] (b) at (6,-5) {\footnotesize $1$};
\node[circle,inner sep =0.4] (c) at (7,-5) {\footnotesize $1$};
\node[circle,inner sep =0.4] (d) at (8,-5) {\footnotesize $1$};
\node[circle,inner sep =0.4] (e) at (13,-3) {\footnotesize $0$};
\node[circle,inner sep =0.4] (f) at (14,-5) {\footnotesize $1$};
\node[circle,inner sep =0.4] (g) at (15,-5) {\footnotesize $1$};
\node[circle,inner sep =0.4] (h) at (16,-5) {\footnotesize $1$};
\node[circle,inner sep =0.4] (i) at (20,-3) {\footnotesize $1$};
\node[circle,inner sep =0.4] (j) at (21,-3) {\footnotesize $1$};
\node[circle,inner sep =0.4] (k) at (22,-5) {\footnotesize $0$};
\node[circle,inner sep =0.4] (l) at (23,-5) {\footnotesize $0$};
\node[circle,inner sep =0.4] (m) at (28,-3) {\footnotesize $s$};
\node[circle,inner sep =0.4] (n) at (29,-3) {\footnotesize $0$};
\node[circle,inner sep =0.4] (o) at (30,-5) {\footnotesize $1$};
\node[circle,inner sep =0.4] (p) at (31,-5) {\footnotesize $0$};
\node[circle,inner sep =0.4] (q) at (32,-5) {\footnotesize $1$};
\draw[->] (b) to[->,bend left=8] (a);
\draw[->] (a) to[->,bend left=20] (d);
\draw[->] (d) to[->,bend left=80,looseness=2] (c);
\draw[->] (c) to[->,bend left=80,looseness=2] (b);
\draw[->] (f) to[->,bend left=8] (e);
\draw[->] (g) to[->,bend left=80,looseness=2] (f);
\draw[->] (h) to[->,bend left=80,looseness=2] (g);
\draw[->] (e) to[->,bend left=8] (i);
\draw[->] (i) to[->,bend left=80,looseness=2] (j);
\draw[->] (j) to[->,bend left=8] (m);
\draw[->] (m) to[->,bend left=80,looseness=1] (k);
\draw (22.5,-5.5)--(23.5,-4.5);
\draw (22.5,-4.5)--(23.5,-5.5);
\draw[->] (k) to[->,bend left=8] (h);
\draw[->,gray] (o) to[->,bend left=8] (m);
\draw[->,gray] (m) to[->,bend left=80,looseness=2] (n);
\draw[->,gray] (p) to[->,bend left=80,looseness=2] (o);
\draw[->,gray] (q) to[->,bend left=80,looseness=2] (p);
\draw[->,gray] (n) to[->,bend left=20] (q);
\end{tikzpicture}
\end{center}
\caption{An illustration of the auxiliary graph structure drawn on a configuration. We have $w_{1,1,3,0111} = 1100010$, $w_{1,2,2,1100} = 11000010$, $w_{0,2,3,10101} = 110000010$, where $0, 1 \in V$ and $0 < 1$, and $\Sigma = \Sigma_0 = \Sigma_1 = \{0,1\}$. The symbol $s$ in the shared cell is interpreted as $0$ when seen from the left, and as $1$ when seen from the right. Edges with color $0$ are shown gray. The cells are shown below their actual positions on the configuration, bottom cells lower than top cells (with the exception of $s$ which is both a top and a bottom cell).}
\label{fig:Construction}
\end{figure}
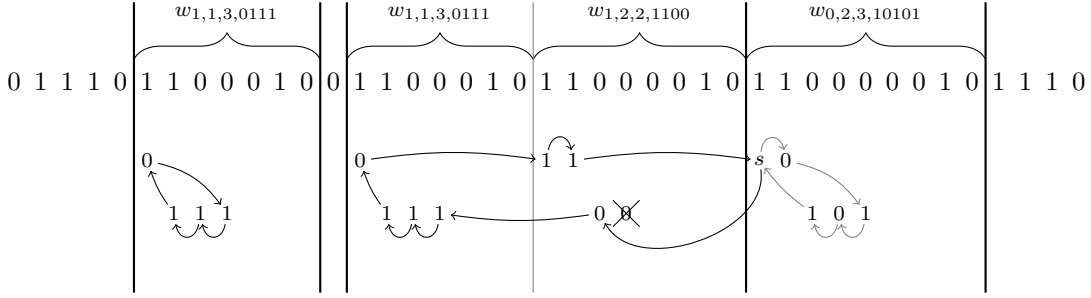

As explained, this drawing of a graph completely describes the mappings $g \mapsto \hat g$ for $g \in G_u$, as soon as the color-$u$ data is given by a continuous shift-commuting procedure for each fixed $u \in V$. Indeed it is: shift-invariance is obvious, and for continuity we observe that the top and bottom pre cells of a conveyor belt of type $u$ are clearly given by a continuous rule, and to get the actual top and bottom cells, we simply need to check whether the converyor belt of type $u$ ends, and if it does, whether a neighboring block exists and is part of a conveyor belt of a lower type $u' < u$, thus we only need to look a finite distance away from the border of the conveyor belt.

It is easy to see that this action gives a homomorphic image of the graph product $G$ corresponding to the groups $(G_u)_{u \in V}$ and the graph $(V,E)$, in the topological full group: for this we simply need to show that whenever $\{u,u'\} \in E$, the images of $G_u$ and $G_{u'}$ commute. This is obvious, since by construction the cycles with edges of colors $u$ and $u'$ are disjoint.

Next we show this is an embedding of $G$, i.e.\ the action of $G$ is faithful. Suppose $w \in (\bigcup_{u \in U} G_u)^*$ is a nontrivial word in the graph group, which is reduced in the standard sense, meaning it has minimal length under the reductions that join the subword $gh \in G_u^2$ to a single symbol $gh \in G_u$, the reduction that removes identity elements, and the commutation relation.

Suppose $|w| > 0$ and write $w$ as
\[ g_{u_m} \circ \cdots \circ g_{u_2} \circ g_{u_1}. \]
Take a maximal sequence of non-commuting elements $g_u$ inside $w$, starting from the right end, i.e.\ take a subword $g_{u_{j_\ell}} \circ \cdots \circ g_{u_{j_2}} \circ g_{u_{j_1}}$ greedily setting $j_1 = 1$ and taking $j_2$ the first one from the right which does not commute with $u_1$, and so on. If everything commutes and $\ell = 1$, then certainly the image in the topological full group acts nontrivially, by considering configurations with a single type-$u$ conveyor belt.

Suppose then $\ell \geq 2$. We show that there exists a word
\[ t = t' t_{j_1} t_{j_2} \cdots t_{j_\ell} t' \]
where each $t_{j_i}$ is a type-$u_{j_i}$ conveyor belt, such that for every $m \geq 1$, every configuration in $[t]_{-|t'|}$ is shifted, after applying $\hat g_{u_m} \cdots \hat g_{u_2} \circ \hat g_{u_1}$, so that the head is in the rightmost shared cell of the conveyor belt $t_{j_m}$.

Whether the shared cell is actually located in $t_{j_i}$ or $t_{j_{i+1}}$ depends on which of $u_{j_i}$ or $u_{j_{i+1}}$ is larger in the arbitrary order $<$, but this does not really matter; to see the situation more symmetrically, it may be helpful to imagine that $\hat g$ acts by moving \emph{precells}, but that we suitably teleport from the shared cell to the corresponding precell. In other words, $\hat g$ can be seen as the element that acts by moving along the conveyor belt on the precells, but conjugated so that one or two precells (at the left and/or right boundary) are moved to a shared cell, and when applying the actions $\hat g_{u_i}$ in order, this conjugating move happens either before or after the application of $\hat g$ depending on the ordering of $V$.

Now, to construct the words $t_{j_i}$, for ${g_{u_{j_i}}} \in G_u$, using the assumption of linear plook-ahead with linearity constant $c_u$, for some $n \geq 1$ pick a periodic point $x$ of period $p \leq 2n+1$ where the cocycle $c$ of $g_{u_{j_i}}$ satisfies
$|c(x)| + \alpha(|c(x)|) \geq n$. 

Now as $t_{j_i}$ we will pick a finite conveyor belt such that the simulated configuration output by $s_u$ on the finite cycle formed by the color-$u_{j_i}$ edges is $x$, the simulated origin of $x$ is in the leftmost top precell of $t_{j_i}$, and the simulated position $p$ in the orbit of $x$ is in the rightmost bottom precell of $t_{j_i}$. This is possible because we have allowed enough flexibility (indeed severe overkill) in the number of top cells and bottom cells included in a single block.

We now explain the allocation of the cell counts more precisely. Let us first massage the numbers $n, p$ picked above: We can take $n \geq 2$ by the choice of the function $f$, and we can take $n > |c(x)|$ by simply choosing a larger $n$ if necessary. Now, by possibly replacing $p$ with $2^r p$ ($x$ is also $2^rp$-periodic), we may assume $|c(x)| < n \leq p-1 \leq 2n+1$. 

Suppose $c(x) = d > 0$, the negative case being symmetric. We think of $x$ as $2p$-periodic, and construct a converyor belt with some number of blocks $a \geq 2$. We take one top cell in the leftmost block, divide $d-1$ top cells evenly between the $a-1$ rightmost blocks, take one bottom cell in the rightmost block, and divide $2p-d-1$ bottom cells evenly between the $a-1$ leftmost blocks. It suffices to check that this is meaningful, as we can then simply lay out $x$ on the cycle as the $2p$-periodic configuration output by $s_u$, as described above.

Now, each cell gets $\lfloor (d-1)/(a-1) \rfloor$ or $\lceil (d-1)/(a-1) \rceil$ top cells, and $\lfloor (2p-d-1)/(a-1) \rfloor$ or $\lceil (2p-d-1)/(a-1) \rceil$ bottom cells. We thus simply have to ensure that
\[ (d-1)/(a-1), (2p-d-1)/(a-1) \in [1,...,100 c_u], \]
equivalently
\[ d-1, 2p-d-1 \in [a-1, ..., 100 (a-1) c_u]. \]


Picking $a = d$, $d-1$ is in this interval, and the latter being in the interval is equivalent to $2p-d-1 \in \{d-1, ..., 100 (d-1) c_u\}$. From the inequalities above we have
\[ 2p-d-1 > p > n \geq d-1. \]
If $n > 2c_u$ then observe that 
$n \leq (1 + c_u)d + c_u$ follows from $|c(x)| + \alpha(|c(x)|) \geq n$, and thus
\[ 2p-d-1 \leq 2p \leq 4n+4 \leq 4(1 + c_u)d + 4c_u + 4 \leq 100 (d - 1) c_u \]

If $n \leq 2c_u$ then $p \leq 4c_u+2$ and we can use a conveyor belt with two blocks: one top cell and $0 < p-d-1 \leq 4 c_u$ bottom cells in the leftmost block; $0 < d-1 < 2c_u$ top cells and one bottom cell in the rightmost block (note that $d \geq 2$ since $d$ is even).

Finally, we show that the look-ahead is (or can be made) linear, and that it is bounded in the case when the graph $(V,E)$ is finite. For this, simply observe that any configuration in $[t]_{-t'}$ is shifted by at least $|t| - m - 2|t'|$ where $m$ is the length of the rightmost block. If the graph $(V,E)$ is finite, this is a bounded look-ahead. If it is infinite, at least $\frac{|t| - m - 2|t'|}{|t|} \geq 1/3$ for large enough $t$ (the worst case being that a maximal length block $m$ is not traversed at the right end, and $t$ has just two blocks). This gives linear look-ahead.
\end{proof}

\section{Embeddings and closure under commensurability}

RAAGs together with closure under passing to finite-index supergroups gives many interesting corollaries, in particular finitely-generated Coxeter groups and surface groups \cite{BeBlMa20}. While closure under commensurability can be shown directly, we extract it from the following stronger theorem. See \cite{LiMa95} for a basic reference on symbolic dynamics.

\begin{theorem}
\label{thm:EmbeddingsBetween}
Let $X$ be a nonwandering sofic shift and $Y$ an uncountable sofic shift. Then $\llb X \rrb \leq \llb Y \rrb$. The embedding preserves the look-ahead function up to a multiplicative constant.
\end{theorem}

\begin{proof}
Let us first massage $X$. We can write a nonwandering $X$ as a finite union $X = \bigcup_i X_i$ where $X_i$ are transitive sofic shifts (which may intersect nontrivially), and $\llb X \rrb$ of course fixes each $X_i$ and acts by a subgroup of $\llb X_i \rrb$. Thus, it is enough to embed $\llb X_i \rrb$ in $\llb Y \rrb$. Letting $Z_i$ be a transitive SFT cover of $X_i$, it is actually enough to embed $\llb Z_i \rrb$: $\llb X_i \rrb$ embeds in $\llb Z_i \rrb$ naturally by acting through the projection. Thus, we may assume $X$ is a transitive SFT in the first place.

We may then suppose that $X \subset \Sigma^\Z$ is a vertex shift by conjugating it to one (changing the alphabet if necessary). Let $r$ be such that any two symbols $a, b \in \Sigma$ that appear in $X$ can be joined by a word of length at most $r$, and pick such words $t_{a, b}$ for each $a, b \in \Sigma$. Pick also a periodic point $x = ...ttt.ttt... \in X$ with period $p$, $|t| = p$. 

Let $Y \subset B^*$. It is standard that we can find in $Y$ a large finite set of arbitrarily long mutually unbordered words $W \subset B^*$ which can be concatenated freely, see e.g.\ \cite{Sa16a} for a proof. Pick such a set $W = \{w_a \;|\;  \in \Sigma\}$, where each $w_a$ is of length at least $1+p+r$.

As in the proof of Theorem~\ref{thm:GraphProducts}, consider maximal runs of consecutive words from $W$, call them \emph{preconveyor belts} and call each occurrence of $w_a \in W$ a \emph{preblock}. We add some further rules: if the subscripts $a$ of the words $w_a$ spell out a forbidden word of $X$, then omit the preblocks included in such a forbidden word. Finally, we erase all conveyor belts of length one. The remaining preblocks are \emph{blocks}, and maximal runs of blocks form \emph{conveyor belts}.

We now name the positions of a block $I$ as follows: the first cell is the \emph{simulating cell}, the following $p$ cells are the \emph{periodic cells}, and the remaining $r$ cells are the \emph{transition cells}. As in the proof of Theorem~\ref{thm:GraphProducts}, we build a graph structure and describe a function $s$ that gives simulated symbols at the nodes, and the action of $\llb X \rrb$ is to act as if the graph drawn is a configuration, with symbols given by $s$.

The edges are as follows: in a conveyor belt, connect the simulating cells of blocks left-to-right, i.e.\ join the simulating cell of a block $I$ to the simulating cell of the block $J$ immediately to the right. Then connect the periodic cells right-to-left. At the left boundary of a conveyor belt, with block $w_a$, let $u$ be a maximally short and lexicographically minimal word such that $tua$ is in the language of $X$, and ``imagine'' that this word connects the periodic cells to the simulating cell, using the transition cells. Now, it is clear what $s$ should output: in the simulating cell of a block $w_a$ it should output $a$, in the periodic cells it outputs $t$ (in reverse), and in the transition cells (that are actually used), it outputs the letters of the transition word $u$. This construction is illustrated in Figure~\ref{fig:Embeddings}.

Clearly the simulated configurations on conveyor belts are ones in $X$, so we indeed have an action of $\llb X \rrb$. Using infinite conveyor belts, we see that the action is faithful.

\begin{figure}[h]
\begin{center}
\begin{tikzpicture}[scale=0.35,shorten >=1pt,shorten <=1pt]
\node () at (0,0) {0};
\node () at (1,0) {0};
\node () at (2,0) {0};
\node () at (3,0) {3};
\node () at (4,0) {2};
\node () at (5,0) {1};
\node () at (6,0) {0};
\node () at (7,0) {3};
\node () at (8,0) {2};
\node () at (9,0) {2};
\node () at (10,0) {0};
\node () at (11,0) {3};
\node () at (12,0) {2};
\node () at (13,0) {2};
\node () at (14,0) {0};
\node () at (15,0) {3};
\node () at (16,0) {2};
\node () at (17,0) {1};
\node () at (18,0) {0};
\node () at (19,0) {3};
\node () at (20,0) {2};
\node () at (21,0) {2};
\node () at (22,0) {0};
\node () at (23,0) {3};
\node () at (24,0) {2};
\node () at (25,0) {1};
\node () at (26,0) {0};
\node () at (27,0) {3};
\node () at (28,0) {2};
\node () at (29,0) {1};
\node () at (30,0) {0};
\node () at (31,0) {3};
\node () at (32,0) {2};
\node () at (33,0) {2};
\node () at (34,0) {0};
\node () at (35,0) {0};
\node () at (36,0) {0};
\node () at (37,0) {0};
%
\draw[thick] (14.5,2) -- (14.5,-7.5);
\draw[gray] (18.5,2) -- (18.5,-7.5);
\draw[gray] (22.5,2) -- (22.5,-7.5);
\draw[gray] (26.5,2) -- (26.5,-7.5);
\draw[gray] (30.5,2) -- (30.5,-7.5);
\draw[thick] (34.5,2) -- (34.5,-7.5);
\draw [decorate,decoration={brace,amplitude=10pt},xshift=0,yshift=0]
(2.5,0.8) -- (6.5,0.8) node [black,midway,yshift=0.6cm] 
{\footnotesize $w_0$};
\draw [decorate,decoration={brace,amplitude=10pt},xshift=0,yshift=0]
(6.5,0.8) -- (10.5,0.8) node [black,midway,yshift=0.6cm] 
{\footnotesize $w_1$};
\draw [decorate,decoration={brace,amplitude=10pt},xshift=0,yshift=0]
(10.5,0.8) -- (14.5,0.8) node [black,midway,yshift=0.6cm] 
{\footnotesize $w_1$};
\draw [decorate,decoration={brace,amplitude=10pt},xshift=0,yshift=0]
(14.5,0.8) -- (18.5,0.8) node [black,midway,yshift=0.6cm] 
{\footnotesize $w_0$};
\draw [decorate,decoration={brace,amplitude=10pt},xshift=0,yshift=0]
(18.5,0.8) -- (22.5,0.8) node [black,midway,yshift=0.6cm] 
{\footnotesize $w_1$};
\draw [decorate,decoration={brace,amplitude=10pt},xshift=0,yshift=0]
(22.5,0.8) -- (26.5,0.8) node [black,midway,yshift=0.6cm] 
{\footnotesize $w_0$};
\draw [decorate,decoration={brace,amplitude=10pt},xshift=0,yshift=0]
(26.5,0.8) -- (30.5,0.8) node [black,midway,yshift=0.6cm] 
{\footnotesize $w_0$};
\draw [decorate,decoration={brace,amplitude=10pt},xshift=0,yshift=0]
(30.5,0.8) -- (34.5,0.8) node [black,midway,yshift=0.6cm] 
{\footnotesize $w_1$};
\draw (4,2)--(5,3);
\draw (5,2)--(4,3);
\draw (8,2)--(9,3);
\draw (9,2)--(8,3);
\draw (12,2)--(13,3);
\draw (13,2)--(12,3);
%
\node[inner sep =0.2] (j) at (15,-2) {\footnotesize $0$};
\node[inner sep =0.2] (k) at (16,-3.4) {\footnotesize $0$};
\node[inner sep =0.2] (kk) at (17,-3.4) {\footnotesize $1$};
\node[circle,fill,inner sep =0] (l) at (18,-5) {\footnotesize $\cdot$};
\node[inner sep =0.2] (m) at (19,-2) {\footnotesize $1$};
\node[inner sep =0.2] (n) at (20,-3.5) {\footnotesize $0$};
\node[inner sep =0.2] (nn) at (21,-3.5) {\footnotesize $1$};
\node[circle,fill,inner sep =0] (o) at (22,-5) {\footnotesize $\cdot$};
\node[inner sep =0.2] (p) at (23,-2) {\footnotesize $0$};
\node[inner sep =0.2] (q) at (24,-3.5) {\footnotesize $0$};
\node[inner sep =0.2] (qq) at (25,-3.5) {\footnotesize $1$};
\node[circle,fill,inner sep =0] (r) at (26,-5) {\footnotesize $\cdot$};
\node[inner sep =0.2] (s) at (27,-2) {\footnotesize $0$};
\node[inner sep =0.2] (t) at (28,-3.5) {\footnotesize $0$};
\node[inner sep =0.2] (tt) at (29,-3.5) {\footnotesize $1$};
\node[circle,fill,inner sep =0] (u) at (30,-5) {\footnotesize $\cdot$};
\node[inner sep =0.2] (v) at (31,-2) {\footnotesize $1$};
\node[inner sep =0.2] (w) at (32,-3.5) {\footnotesize $0$};
\node[inner sep =0.2] (ww) at (33,-3.5) {\footnotesize $1$};
\node[inner sep =0.2] (x) at (34,-5) {\footnotesize 0};
\draw[->] (j) to[bend left=8] (m);
\draw[->] (m) to[bend left=8] (p);
\draw[->] (p) to[bend left=8] (s);
\draw[->] (s) to[bend left=8] (v);
\draw[->] (ww) to[bend left=80,looseness=2] (w);
\draw[->] (w) to[bend left=8] (tt);
\draw[->] (tt) to[bend left=80,looseness=2] (t);
\draw[->] (t) to[bend left=8] (qq);
\draw[->] (qq) to[bend left=80,looseness=2] (q);
\draw[->] (q) to[bend left=8] (nn);
\draw[->] (nn) to[bend left=80,looseness=2] (n);
\draw[->] (n) to[bend left=8] (kk);
\draw[->] (kk) to[bend left=80,looseness=2] (k);
\draw[->] (v) to[bend left=40,looseness=1.5] (x);
\draw[->] (x) to[bend left=30,looseness=1] (ww);
\draw[->] (k) to[bend left=30,looseness=1] (j);
%
\end{tikzpicture}
\end{center}
\caption{The nodes and connections described in the proof of Theorem~\ref{thm:EmbeddingsBetween}, where $Y = \{0,1,2\}^\Z$, $X \subset \{0,1\}^\Z$ is the golden mean shift with a single forbidden word $11$ so $r = 1$, $w_0 = 3210, w_1 = 3220$, and $x = ...0101.0101...$ so $p = 2$. We use a different vertical offset for the simulating cell, the periodic cells and the transition cell, for a clearer drawing. Numbers in the cells indicate the simulating symbols. A forbidden word, and a resulting isolated block have been eliminated. Self-loops and nodes with only self-loops are omitted, except for unused transition cells which are shown with black dots.}
\label{fig:Embeddings}
\end{figure}
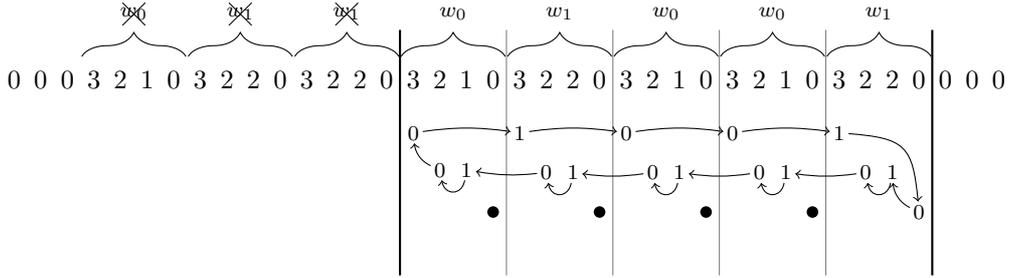

The claim about the look-ahead function is clear; we simply need enough overhead in the look-ahead function to parse the conveyor belt structure.
\end{proof}

The theorem is optimal in some (but not all) ways: If $Y$ is countable, then $\llb Y \rrb$ is elementary amenable (this is well-known, one proof is in \cite{SaSc16a}), thus cannot contain $\llb \Sigma^\Z \rrb$ for any alphabet $\Sigma$ (thus cannot contain $\llb X \rrb$ for any uncountable sofic shift $X$ by the above theorem). If $X$ is a wandering subshift, it is easy to see that it contains a copy of the infinite alternating group of $\N$ (by permuting positions around a wandering clopen set), and thus cannot embed in $\llb Y \rrb$ for any subshift with dense periodic points (such as a nonwandering sofic shift) $Y$, since the topological full group of such $Y$ is residually finite. This does not deal with cases where $X$ and $Y$ are both wandering, and we suspect that the classification of such pairs is more difficult.

The corresponding embeddability problem for automorphism groups of subshifts $\Aut(X)$ is open, even when $X$ and $Y$ are mixing SFTs, with the expection of the case where $X$ is a full shift (and some specially crafted examples). Embeddability of full shifts is shown in \cite{KiRo90,Sa16a}.

\begin{corollary}
If $G$ is commensurable to $H$ and $H \in \mathcal{G}$, then $G \in \mathcal{G}$. The look-ahead of $H$ differs from that of $G$ by an multiplicative constant.
\end{corollary}

\begin{proof}
It is enough to show this for $G$ a supergroup of $H$. If $X \subset \Sigma^\Z$ is a subshift and $\# \notin \Sigma$ a symbol not in its alphabet, write $\sqrt[k]{X}$ for the smallest subshift containing all $x \in (\Sigma \cup \{\#\})^\Z$ such that 
\[ (x_{ki})_i \in X \wedge \forall i: \forall j \in [1, k-1]: x_{ki + j} = \#. \]
This is just a suspension of $X$ over a finite cycle, meaning the dynamics is a $k$th root of the diagonal dynamics on a disjoint union of $k$ copies of $X$. Clearly the wreath product $\llb X \rrb \wr S_k$ embeds naturally in $\llb \sqrt[k]{X} \rrb$, and thus any index-$k$ supergroup of $\llb X \rrb$ embeds in $\llb \sqrt[k]{X} \rrb$ by induced representation \cite{KiRo90}. Of course $\llb \sqrt[k]{X} \rrb$ is nonwandering so by the previous theorem we have closure under commensurability.
\end{proof}

As explained in the journal version of \cite{BaKaSa16}, $\llb \sqrt[k]{\Sigma^\Z} \rrb$ is naturally isomorphic to the group of $k$-headed finite-state automata $\mathrm{RFA}(|\Sigma|,k)$ defined in \cite{BaKaSa16}.

\section{Lamplighter groups}

Theorem~\ref{thm:GraphProducts} would be more interesting if we had more exotic examples of groups acting with linear look-ahead. We suspect that there are groups that act with linear look-ahead and cannot be constructed from scratch with virtual extensions and graph products, but we have no candidates.

We also suspect that some groups act with only non-linear look-ahead. For this, we suggest the following simple candidate:

\begin{conjecture}
The lamplighter group $\Z_2 \wr \Z$ does not embed in the topological full group of a full shift with linear look-ahead.
\end{conjecture}

It is not difficult to find an embedding of $\Z_2 \wr \Z$ with non-linear look-ahead. We prove a bit more, and study conditions under which wreath products $A \wr G$ are in $\mathcal{G}$.

\begin{definition}
If $A$ is a finite abelian group, we say an action of $G \leq \llb X \rrb$ is \emph{move-$A$ithful} if the following hold: Let $c_g$ be the cocycle of $g \in G$. Then for every mapping $\beta : G \to A$ with finite nonempty support, there exists $x \in \Sigma^\Z$ and $\gamma : \Z \to \End(A)$ such that $\sum_{g \in G} \gamma(c_g(x))(\beta(g)) \neq 0_A$.
\end{definition}

Move-$A$ithfulness roughly corresponds to $A$ithfulness as defined in \cite{Sa20}; this latter property applies to any action on a zero-dimensional space, in particular to topological full groups. We omit the discussion of the precise connection of the two notions.


\begin{definition}
Let $G \in \mathcal{G}$. We say an action of $G \leq \llb X \rrb$ has \emph{unique moves} if for every $\emptyset \neq F \Subset G$, there exists $g \in F$ with cocycle $c$, and a point $x$ such that $c'(x) \neq c(x)$ whenever $c'$ is the cocycle of an element $h \in G \setminus \{g\}$.
\end{definition}

Unique moves roughly correspond to strong faithfulness as defined in \cite{Sa20}. The following result is analogous to Lemma~3 in \cite{Sa20} and we omit the proof. The idea is simply to use the identity endomorphism at the position reached by a unique element of the support, and the zero endomorphism elsewhere.

\begin{lemma}
An action with unique moves is move-$A$ithful for any finite abelian $A$.
\end{lemma}

\begin{theorem}
\label{thm:LamplighterEmbedding}
Let $A$ be a nontrivial finite abelian group, $G \leq \llb X \rrb$ act move-$A$ithfully, and $\Sigma$ be an alphabet. Then $A \wr G \leq \llb \Sigma^\Z \rrb$.
\end{theorem}

The following proof could be simplified by using $\sqrt[4]{\{0,1\}^\Z}$ and applying Theorem~\ref{thm:EmbeddingsBetween} instead of explicit conveyor belts, but we prefer to give a concrete direct construction for the basic lamplighter group $\Z_2 \wr \Z$.

\begin{proof}
We give the proof first for $A = \Z_2$ and $G = \Z$, and explain the modifications for the general result. By Theorem~\ref{thm:GraphProducts}, it is enough to show the result for $\Sigma = \{0,1,2,3\}$. Let $t$ be the generator for $\Z$ and $a$ the involution. Define $u = 3210$, $v = 3220$. As in the proof of Theorem~\ref{thm:GraphProducts}, to a configuration $x \in \Sigma^\Z$ we associate a graph with a node at each $n \in \Z$, and edges labeled with $t$ and $a$. Nodes that are not part of an occurrence of $u$ or $v$ are fixed by $t$ and $a$.

Nodes in maximal runs of words $w \in \{u, v\}$ are connected according to Figure~\ref{fig:LamplighterConveyors}. In words, if $I$ and $J$ are consecutive blocks containing words from $\{u, v\}$, add a $t$-edge from the first and second positions of $I$ to the respective positions of $J$, and add $t$-edges from the third and fourth position of $J$ to the respective positions of $I$. At the ends, wrap up the connections. Finally, inside every occurrence of $v$ at some interval $I$, add an $a$-edge between the first and second position.

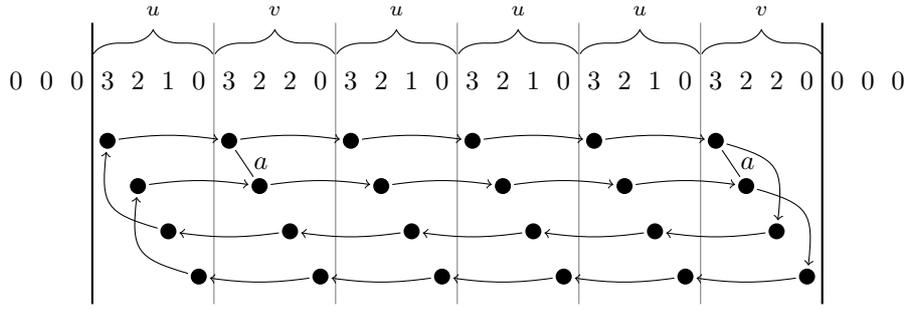
\begin{figure}
\begin{center}
\begin{tikzpicture}[scale=0.4,shorten >=1pt,shorten <=1pt]
\node () at (0,0) {0};
\node () at (1,0) {0};
\node () at (2,0) {0};
\node () at (3,0) {3};
\node () at (4,0) {2};
\node () at (5,0) {1};
\node () at (6,0) {0};
\node () at (7,0) {3};
\node () at (8,0) {2};
\node () at (9,0) {2};
\node () at (10,0) {0};
\node () at (11,0) {3};
\node () at (12,0) {2};
\node () at (13,0) {1};
\node () at (14,0) {0};
\node () at (15,0) {3};
\node () at (16,0) {2};
\node () at (17,0) {1};
\node () at (18,0) {0};
\node () at (19,0) {3};
\node () at (20,0) {2};
\node () at (21,0) {1};
\node () at (22,0) {0};
\node () at (23,0) {3};
\node () at (24,0) {2};
\node () at (25,0) {2};
\node () at (26,0) {0};
\node () at (27,0) {0};
\node () at (28,0) {0};
\node () at (29,0) {0};
\draw[thick] (2.5,2) -- (2.5,-7.5);
\draw[gray] (6.5,2) -- (6.5,-7.5);
\draw[gray] (10.5,2) -- (10.5,-7.5);
\draw[gray] (14.5,2) -- (14.5,-7.5);
\draw[gray] (18.5,2) -- (18.5,-7.5);
\draw[gray] (22.5,2) -- (22.5,-7.5);
\draw[thick] (26.5,2) -- (26.5,-7.5);
\draw [decorate,decoration={brace,amplitude=10pt},xshift=0,yshift=0]
(2.5,0.8) -- (6.5,0.8) node [black,midway,yshift=0.6cm] 
{\footnotesize $u$};
\draw [decorate,decoration={brace,amplitude=10pt},xshift=0,yshift=0]
(6.5,0.8) -- (10.5,0.8) node [black,midway,yshift=0.6cm] 
{\footnotesize $v$};
\draw [decorate,decoration={brace,amplitude=10pt},xshift=0,yshift=0]
(10.5,0.8) -- (14.5,0.8) node [black,midway,yshift=0.6cm] 
{\footnotesize $u$};
\draw [decorate,decoration={brace,amplitude=10pt},xshift=0,yshift=0]
(14.5,0.8) -- (18.5,0.8) node [black,midway,yshift=0.6cm] 
{\footnotesize $u$};
\draw [decorate,decoration={brace,amplitude=10pt},xshift=0,yshift=0]
(18.5,0.8) -- (22.5,0.8) node [black,midway,yshift=0.6cm] 
{\footnotesize $u$};
\draw [decorate,decoration={brace,amplitude=10pt},xshift=0,yshift=0]
(22.5,0.8) -- (26.5,0.8) node [black,midway,yshift=0.6cm] 
{\footnotesize $v$};
\node[circle,fill,inner sep =0.2] (a) at (3,-2) {\footnotesize a};
\node[circle,fill,inner sep =0.2] (b) at (4,-3.5) {\footnotesize a};
\node[circle,fill,inner sep =0.2] (c) at (5,-5) {\footnotesize a};
\node[circle,fill,inner sep =0.2] (d) at (6,-6.5) {\footnotesize a};
\node[circle,fill,inner sep =0.2] (e) at (7,-2) {\footnotesize a};
\node[circle,fill,inner sep =0.2] (f) at (8,-3.5) {\footnotesize a};
\node[circle,fill,inner sep =0.2] (g) at (9,-5) {\footnotesize a};
\node[circle,fill,inner sep =0.2] (h) at (10,-6.5) {\footnotesize a};
\node[circle,fill,inner sep =0.2] (i) at (11,-2) {\footnotesize a};
\node[circle,fill,inner sep =0.2] (j) at (12,-3.5) {\footnotesize a};
\node[circle,fill,inner sep =0.2] (k) at (13,-5) {\footnotesize a};
\node[circle,fill,inner sep =0.2] (l) at (14,-6.5) {\footnotesize a};
\node[circle,fill,inner sep =0.2] (m) at (15,-2) {\footnotesize a};
\node[circle,fill,inner sep =0.2] (n) at (16,-3.5) {\footnotesize a};
\node[circle,fill,inner sep =0.2] (o) at (17,-5) {\footnotesize a};
\node[circle,fill,inner sep =0.2] (p) at (18,-6.5) {\footnotesize a};
\node[circle,fill,inner sep =0.2] (q) at (19,-2) {\footnotesize a};
\node[circle,fill,inner sep =0.2] (r) at (20,-3.5) {\footnotesize a};
\node[circle,fill,inner sep =0.2] (s) at (21,-5) {\footnotesize a};
\node[circle,fill,inner sep =0.2] (t) at (22,-6.5) {\footnotesize a};
\node[circle,fill,inner sep =0.2] (u) at (23,-2) {\footnotesize a};
\node[circle,fill,inner sep =0.2] (v) at (24,-3.5) {\footnotesize a};
\node[circle,fill,inner sep =0.2] (w) at (25,-5) {\footnotesize a};
\node[circle,fill,inner sep =0.2] (x) at (26,-6.5) {\footnotesize a};
\draw[->] (a) to[bend left=8] (e);
\draw[->] (b) to[bend left=8] (f);
\draw[<-] (c) to[bend right=8] (g);
\draw[<-] (d) to[bend right=8] (h);
\draw[->] (e) to[bend left=8] (i);
\draw[->] (f) to[bend left=8] (j);
\draw[<-] (g) to[bend right=8] (k);
\draw[<-] (h) to[bend right=8] (l);
\draw[->] (i) to[bend left=8] (m);
\draw[->] (j) to[bend left=8] (n);
\draw[<-] (k) to[bend right=8] (o);
\draw[<-] (l) to[bend right=8] (p);
\draw[->] (m) to[bend left=8] (q);
\draw[->] (n) to[bend left=8] (r);
\draw[<-] (o) to[bend right=8] (s);
\draw[<-] (p) to[bend right=8] (t);
\draw[->] (q) to[bend left=8] (u);
\draw[->] (r) to[bend left=8] (v);
\draw[<-] (s) to[bend right=8] (w);
\draw[<-] (t) to[bend right=8] (x);
\draw[->] (c) to[bend left=40,looseness=1.5] (a);
\draw[->] (d) to[bend left=40,looseness=1.5] (b);
\draw[->] (u) to[bend left=40,looseness=1.5] (w);
\draw[->] (v) to[bend left=40,looseness=1.5] (x);
\draw[] (e) -- node[right] {$a$} (f);
\draw[] (u) -- node[right] {$a$} (v);
\end{tikzpicture}
\end{center}
\caption{The nodes and connections described in the proof of Theorem~\ref{thm:LamplighterEmbedding}, in the case $A = \Z_2, G = \Z$. The four nodes under each $w \in \{u, v\}$ are drawn with vertical offsets for a clearer drawing. The $t$-edges are unlabeled, and self-loops and nodes with only self-loops are omitted.}
\label{fig:LamplighterConveyors}
\end{figure}

Clearly $a^2$ acts as the identity. To see that $[a, a^{t^i}]$ acts as identity, observe that the nodes are paired up in a natural way, with, under each $w \in \{u, v\}$, the first and second position forming a pair, and the third and fourth forming a pair. The action of $t$ is to move the head around a finite conveyor belt, or along an infinite path, and a pair never separates and keeps the same relative offset. The action of $a$ exchanges a pair if the current block is of type $v$. The action of $a^{t^i}$ then always returns the head either back to the original node or to its pair, and this clearly implies $[a, a^{t^i}] = \ID$, so we have an action of the lamplighter group. The action on configurations with one infinite conveyor belt and a single occurrence of $v$ is easily seen to prove faithfulness.

The general idea is now very similar to the case of wreath products in \cite{Sa20}, we only outline it here. The essential feature of $A$ is abelianity: we can replace the pairs with $|A|$-tuples, and have $A$ act by its regular action when the head is in a marked block ($v$-blocks are marked, and $u$-blocks are not), and trivially otherwise. We call the position of the head inside an $|A|$-tuple its \emph{$A$-state}.

Now note that commutation of the conjugates $a^{t^i}$ does not actually care what the actions of the elements $t^i$ is, as long as the relative positions of cells in a pair stay the same. Thus, we can replace the action of $\langle t \rangle$ by any action $G \leq \llb \Sigma^\Z \rrb$, and we will get a well-defined action of $A \wr G$. For this we simply need to add some symbol information (an element of $\Sigma^2$) in each of the blocks $w_0, w_1$ (in addition to, not replacing, the bit in the subindex) and wrap them to a conveyor belt as in the proof of Theorem~\ref{thm:GraphProducts}.

For faithfulness of this action, we used the unique moves property of the shift action of $\Z$. It is indeed enough to have an action which is move-$A$ithful. To see this, replace the subscripts in $w_0$ and $w_1$ with arbitrary endomorphisms of $A$ (in addition to the added symbol information described in the previous paragraph), and when the head is in in a cell carrying an endomorphism $\phi : A \to A$, the action of $\hat h$ for $h \in A$ sums $\phi(h)$ to the $A$-state. The definition of move-$A$ithfulness precisely states that on free configurations we can set up the endomorphisms so that an element where the total $G$-action is trivial will act on the state non-trivially. Of course a non-trivial $G$ action is visible in the movement of the head (by the assumption that the original $G$ action is even faithful), so we indeed have faithfulness of the $A \wr G$-action.
\end{proof}

One could generalize the theorem by allowing actions without move-$A$ithfulness, understanding that the wreath product will be with respect to some non-free action of the group $G$, i.e.\ we get some group $A \wr_\Omega G$ where $G \curvearrowright \Omega$.

As shown in the proof, $\Z$ even acts with unique moves, and finite groups are also easily seen to act with unique moves (permute the positions under an unbordered word). It is easy to show that if countably many groups $(G_i)_i$ act with unique moves, then their free product does too, by the construction in Theorem~\ref{thm:GraphProducts}. We obtain the following theorem.

\begin{theorem}
Let $G$ be any countable free product of finite groups and copies of $\Z$. Then $A \wr G \in \mathcal{G}$ for any finite abelian group $A$.
\end{theorem}

We have a simple example for a group that embeds in $\llb \Sigma^\Z \rrb$ (with linear look-ahead) but does not admit any move-$A$ithful actions:

\begin{theorem}
The group $\Z^2$ does not embed in $\llb \Sigma^\Z \rrb$ with a move-$A$ithful action, for any alphabet $\Sigma$ and nontrivial finite abelian group $A$.
\end{theorem}

\begin{proof}
Write $G = \Z^2$ multiplicatively and $A$ additively. Let $h \in A \setminus \{0_H\}$. Note that $G$ has superlinear growth, so there exists a finite set $1 \notin F \Subset G$ such that $\forall x \in \Sigma^\Z: \exists g \in F: gx = x$. Since $G$ is abelian, $\mathrm{stab}_G(x) = \mathrm{stab}_G(gx)$ for all $g \in G$. Because aperiodic points are dense in $\Sigma^\Z$ (and the action determines the cocycle values on all aperiodic $\sigma$-orbits), for all $x \in \Sigma^\Z$ we have $c_g(x) = 0$ for some $g \in F$.

Let $F = \{g_1, g_2, ..., g_k\}$ and for $I \subset \{1,2,...,k\}$ define $g_I = \prod_{i \in I} g_i$. Now define $\beta : G \to A$ by
\[ \beta(g) = \sum \{ (-1)^{|I|} \cdot h \;|\; g_I = g \} \]
where $(-1)^{|I|} \cdot h$ means we take $h$ or $-h$ depending on the parity of $|I|$. Note that $\beta(g) = 0$ if $g \neq g_I$ for all $I$, so this has finite support. Note also that $\beta$ is not the zero map, since taking any $g_i$ and an open half-plane around it, the product of all elements $g_j$ on that half-plane has a unique representation as $g_I$.

Now, for any $x \in \Sigma^\Z$ and $\gamma : \Z \to \End(A)$ if $c_{g_i}(x) = 0$ then
\begin{align*}
\sum_{g \in G} \gamma(c_g(x))(\beta(g)) &= \sum_{I \subset \{1,2,...,k\}} \gamma(c_{g_I}(x))((-1)^{|I|} h) \\
&= \sum_{I \subset \{1,2,...,k\}, i \notin I} (\gamma(c_{g_I}(x))((-1)^{|I|} \cdot h) + \gamma(c_{g_{I \cup \{i\}}}(x))((-1)^{|I| + 1} \cdot h))
\end{align*}
As observed above, by the abelianity of $G$, $c_{g_i}(x) = 0$ implies $c_{g_i}(g(x)) = 0$ for any $g \in G$, and thus for $i \notin I$ we have $c_{g_I}(x) = c_{g_{I \cup \{i\} } }(x)$, thus
\[ \gamma(c_{g_I}(x))((-1)^{|I|} \cdot h) = -\gamma(c_{g_{|I| \cup \{i\}}}(x))((-1)^{|I|+1} \cdot h), \]
and the sum cancels to zero. This contradicts move-$A$ithfulness
\end{proof}

Of course it follows that an abelian group $G$ admits a move-$A$ithful action for some non-trivial $A$ if and only if it admits an action with unique moves, if and only if it is virtually cyclic. The previous theorem shows that the construction Theorem~\ref{thm:LamplighterEmbedding} cannot be applied to the group $\Z_2 \wr \Z^2$. It seems non-trivial to prove that no other construction will work, although we believe this is the case.

\begin{conjecture}
Let $A$ be a nontrivial finite abelian group, $\Sigma$ an alphabet. Then $A \wr \Z^d \not\leq \llb \Sigma^\Z \rrb$ if $d \geq 2$.
\end{conjecture}

\bibliographystyle{plain}
\bibliography{bib}{}

\end{document}